\documentclass[11pt,a4paper]{amsart}
\usepackage{fullpage}
\usepackage{amssymb,amscd,hyperref}
\usepackage[all]{xy}%\CompileMatrices

\newtheorem{thm}{Theorem}[section]
\newtheorem{lem}[thm]{Lemma}
\newtheorem{prop}[thm]{Proposition}
\newtheorem{cor}[thm]{Corollary}

\theoremstyle{definition}
\newtheorem{rem}[thm]{Remark}
\newtheorem{defn}[thm]{Definition}

\def\Z{\mathbb{Z}}
\def\Q{\mathbb{Q}}
\def\F{\mathbb{F}}

\def\scom{s\mathcal{C}\mathrm{om}_\Q}
\def\dcom{d\mathcal{C}\mathrm{om}_\Q}

\def\leq{\leqslant}
\def\geq{\geqslant}

\def\harr{{\sf Harr}}
\def\Com{{\sf Com}}

\def\Sh{\sf{Sh}}
\def\hot{\hat{\otimes}}
\def\hodot{\hat{\odot}}
\def\ra{\rightarrow}

\def\AQ{{\sf AQ}}

\def\Sp{\text{$\sf{Sp}$}}
\def\modules{\text{$\sf{mod}$}}
\def\smod{\text{$s\sf{mod}$}}
\def\dgmod{\text{dg$\sf{mod}$}}

\def\sigmamod{\text{$\sf{mod}\Sigma$}}

\def\nonnegch{\text{$\sf{Ch}_{\geq 0}$}}
\def\sigmadg{\text{$\sf{dg}\Sigma$ }}
\def\sigmas{\text{$\sf{s}\Sigma$ }}
\def\sigmadgred{\text{$\sf{dg}\Sigma_+$}}
\def\comsigmadg{C^+_\sigmadg}
\def\comsigmas{C^+_\sigmas}

\DeclareMathOperator{\colim}{colim}
\def\ie{\emph{i.e.}}
\def\eg{\emph{e.g.}}
\def\id{\mathrm{id}}

\def\t{\odot}

\begin{document}
\title{
On the homology and homotopy of commutative shuffle algebras}
\author{Birgit Richter}
\address{Fachbereich Mathematik der Universit\"at Hamburg,
Bundesstra{\ss}e 55, 20146 Hamburg, Germany}
\email{richter@math.uni-hamburg.de}
\urladdr{http://www.math.uni-hamburg.de/home/richter/}
\date{\today}

\begin{abstract}
For commutative algebras there are three important homology
theories, Harrison homology, Andr\'e-Quillen homology and
Gamma-homology. In general these differ, unless one works with
respect to a ground field of characteristic zero. We show that the
analogues of these homology theories agree in the category of
 pointed commutative monoids in symmetric sequences and that
 Hochschild homology always possesses a Hodge decomposition in this
 setting. In addition we prove that the category of pointed 
differential graded commutative monoids in symmetric sequences has a
model structure and that it is Quillen equivalent to the model 
category of pointed simplicial commutative monoids in symmetric sequences.  
\end{abstract}
\maketitle

\section{Introduction}

Symmetric sequences are used in many areas of mathematics: Joyal
\cite{Joyal} 
baptized them \emph{species}, they are important in the theory of
operads and one symmetric monoidal model of the stable homotopy
category is given by symmetric spectra \cite{hss} and these are spectra
that are built out of symmetric sequences in spaces (simplicial
sets).

Let $k$ be an arbitrary commutative ring with unit. The aim of this
paper is to understand homology theories for commutative monoids in
symmetric sequences and to investigate the corresponding homotopy
theory of simplicial and differential graded objects. A crucial
ingredient for our investigation is a result by 
Stover. He proved \cite[9.10]{St} that the
norm map $N = \sum_{\sigma \in \Sigma_n}\sigma $ induces an
isomorphism between coinvariants and invariants of tensor powers 
\begin{equation} \label{eq:norm}
N \colon (V^{\t n})_{\Sigma_n} \rightarrow (V^{\t n})^{\Sigma_n}
\end{equation}
for all $n \geq 1$ and all reduced symmetric sequences $V$, thus the
zeroth Tate
cohomology group of $\Sigma_n$ with coefficients in $V^{\otimes n}$
vanishes. Therefore,  
quite often arguments work in this context that otherwise only hold in
characteristic zero. A second way to interpret Stover's result is that
the difference between commutative monoids and commutative monoids
with divided power structures disappears. In fact we show that there
is a natural way to associate an ordinary graded divided power algebra
to every commutative monoid in symmetric sequences (see Theorem \ref{thm:pd}).

For commutative algebras $A$ over a commutative ring $k$ there are several
important homology theories: Andr\'e-Quillen homology, $\AQ_*$,  takes derived
functors of indecomposables in the simplicial sense, Harrison
homology, $\harr_*$,
takes the indecomposables of the Hochschild complex and
$\Gamma$-homology, $H\Gamma_*$, views the commutative algebra as an
$E_\infty$-algebra and takes its homology in this setting. An
identification \cite{PR} says that $\Gamma$-homology can also we
viewed as the stabilization of a commutative algebra: The category of
commutative augmented algebras is enriched in the category of pointed simplicial
sets, so $A \otimes \mathbb{S}^n$ makes sense for a simplicial model
$\mathbb{S}^n$ of the $n$-sphere. Taking $\mathbb{S}^n =
(\mathbb{S}^1)^{\wedge n}$ gives rise to natural stabilization maps 
$\pi_k(A\otimes \mathbb{S}^n) \ra \pi_{k+1}(A \otimes
\mathbb{S}^{n+1})$ and $\Gamma$-homology can be identified with the
stable homotopy groups of $A \otimes \mathbb{S}^\bullet$ \cite[Theorem
1]{PR}.

In general
all three homology theories differ drastically: if $k$ is $\F_2$ for
instance, then Andr\'e-Quillen homology of the polynomial ring
$\F_2[x]$ vanishes above degree zero whereas $\Gamma$-homology of
$\F_2[x]$ with coefficients in $\F_2$ is isomorphic to
$(H\F_2)_*H\Z$ \cite[3.2]{riro}. In contrast Harrison homology of $\F_2$
viewed as an 
$\F_2$-algebra is not trivial in positive degrees whereas Andr\'e-Quillen
homology and $\Gamma$-homology are zero (see for instance
\cite[Example 6.7]{RoWh}).  One reason for these 
differences is that the operad of commutative algebras is in general not
$\Sigma$-free unless we work in characteristic zero: the operad $\Com$
over $k$ is $\Com(n) = k$ for all $n$
and this is $k$-projective but not projective as a module over the
group algebra of the symmetric group $\Sigma_n$,
$k[\Sigma_n]$. Usually one replaces the operad $\Com$ by an
$E_\infty$-operad to make things homotopy invariant. For instance Mike
Mandell showed \cite[1.8]{M} that the normalization functor induces an
isomorphism  between Andr\'e-Quillen homology for  simplicial 
$E_\infty$-algebras and Andr\'e-Quillen homology for 
differential graded $E_\infty$-algebras. We take the
point of view that we want to keep the operad $\Com$ but work in an underlying
symmetric monoidal category with enough $\Sigma$-freeness to make up
for the algebraic defects of the operad $\Com$. 

We show
that  Harrison homology, Andr\'e-Quillen
homology and Gamma homology coincide if we consider reduced
commutative monoids in symmetric sequences. In other contexts such
algebras are called commutative shuffle algebras \cite{Ro}; shuffle
algebras play an important r\^ole in the theory of combinatorial Hopf
algebras. 
If  $A$ is such a monoid and if $A$ is levelwise projective as a $k$-module we
have
$$ AQ_*(A;k) \cong \Sigma^{-1}\harr_*(A) \cong H\Gamma_*(A)$$
(see Theorems \ref{thm:aqharr} and \ref{thm:aqhgamma}). We also show
that the Hodge 
decomposition for Hochschild homology is valid in our context in 
arbitrary characteristic (Theorem \ref{thm:hodge}). 

We establish a model category structure for pointed commutative
monoids in symmetric sequences of chain complexes with fibrations and
weak equivalences induced by the underlying category
(see Corollary \ref{cor:mcat}). This should be 
compared to commutative differential graded algebras, where such a
model structure does not exist unless one works over the rationals. 

We extend the classical Dold-Kan 
correspondence to a Quillen equivalence between the model category
structures of pointed simplicial commutative shuffle algebras and
pointed differential graded commutative shuffle algebras (Theorem
\ref{thm:doldkancommsh}). Here we call 
a monoid $A$ pointed if its zeroth level consists precisely of the
unit of the underlying category. This generalizes Quillen's result
\cite[Remark on p.~223]{Q69} in the characteristic zero setting. 

Brooke Shipley showed \cite{shipley} that there is a Quillen
equivalence between the model categories of $Hk$-algebra spectra and
differential graded $k$-algebras for any commutative ring $k$. We use
her chain of functors to show in Proposition \ref{prop:hkexamples}
that commutative $Hk$-algebra spectra 
give rise to natural examples of differential graded and simplicial commutative
shuffle algebras. In fact, we conjecture that there is a Quillen
equivalence between commutative $Hk$-algebra spectra and commutative
monoids in spectra of simplicial $k$-modules and commutative monoids
in spectra of differential graded $k$-modules. We plan to address
this question in future work. 

\textbf{Acknowledgement}: I thank Teimuraz Pirashvili for several helpful
  discussions. The idea, that all common homology theories of
  commutative algebras should agree in the setting of symmetric sequences is
  due to him. John Rognes asked whether commutative shuffle algebras
  are divided power algebras in a suitable sense and thanks to this
  question I thought about the structures that lead to Theorem
  \ref{thm:pd}. 

\section{Commutative shuffle algebras}

We will fix an arbitrary commutative ground ring $k$. If $S$ is a set, 
then $kS$ denotes the free $k$-modules generated by $S$. We denote by 
$\Sigma_n$ the symmetric group on $n$ letters and by $k\Sigma_n$ its 
group algebra.

For general facts about symmetric sequences I recommend \cite{AM,St}.
If $\mathcal{C}$ is any category, then we denote by
$\mathcal{C}\Sigma$ the category of symmetric sequences in
$\mathcal{C}$. Its objects are sequences $X=(X(\ell))_{\ell \in
  \mathbb{N}_0}$ of objects $X(\ell)$ of
$\mathcal{C}$ such that every $X(\ell)$ carries a left
$\Sigma_\ell$-action with $\Sigma_\ell$ denoting the symmetric group
on $\ell$ letters. We call $X(\ell)$ the $\ell$th level of $X$. A
morphism $f\colon X \ra Y$ between two objects $X,Y$ of
$\mathcal{C}\Sigma$ is a sequence of morphisms $f(\ell)\colon X(\ell)
\ra Y(\ell)$ in $\mathcal{C}$ such that every $f(\ell)$ is 
$\Sigma_\ell$-equivariant.

If $(\mathcal{C},\boxtimes, 1)$ is symmetric monoidal and possesses
sums that distribute over the monoidal product, \eg, if $\mathcal{C}$
is closed, then $\mathcal{C}\Sigma$ inherits a symmetric monoidal
product from $\mathcal{C}$ if $\mathcal{C}$ has coequalizers.

In the following the categories of $k$-modules, simplicial modules and
chain complexes play an important r{\^o}le. If we denote the category
of $k$-modules by $\modules$, then the usual tensor product of
$k$-modules, $\otimes$, gives rise to a symmetric monoidal category
$(\modules\Sigma,\odot,I)$ such that for two $M,N \in \modules\Sigma$
we get in level $\ell$
$$ (M \odot N)(\ell) = \bigoplus_{p+q=\ell} k\Sigma_\ell
\otimes_{k\Sigma_p \otimes k\Sigma_q} M(p) \otimes N(q).$$ Here the
right-hand side denotes the coequalizer with respect to the
$\Sigma_p\times \Sigma_q$-action on $M(p)\otimes N(q)$ and on 
$\Sigma_\ell$ by viewing $\Sigma_p\times \Sigma_q$ as a subgroup of
$\Sigma_\ell = \Sigma_{p+q}$. This structure is usually called the
tensor product of symmetric sequence of modules. We use the notation
$\odot$ in an attempt to minimize confusion. 

The unit for this symmetric monoidal structure is the symmetric
sequence $I$ with
$$ I(\ell) = 0 \text{ for } \ell \neq 0 \text{ and } I(0) =k.$$
For the symmetry we denote by $\chi(q,p)\in \Sigma_{q+p}$ the permutation with
$$ \chi(q,p)(i) = \begin{cases} i+p, \text{ for } 1 \leq i \leq q, \\
                                i-q, \text{ for } q < i \leq
                                q+p.   \end{cases}$$

The symmetry, $tw$,  for $\odot$ is then defined by sending a class
$[\sigma \otimes x \otimes y]$ with $\sigma\in \Sigma_{p+q}$, $x \in
X(p)$ and $y\in Y(q)$ to
$$ tw[\sigma \otimes x\otimes y] = [(\sigma\circ \chi(q,p)) \otimes y
\otimes x].$$

For graded modules the situation is similar but we might introduce the
sign $(-1)^{|x||y|}$ in $tw$. These symmetric monoidal structures on $\modules$
and 
$\modules\Sigma$ transfer to symmetric monoidal structures  on
non-negatively graded chain 
complexes, $\dgmod$, and to $\dgmod\Sigma$ such that for $X_*,Y_* \in
\dgmod\Sigma$ (with $*$ denoting the chain degree) we get in chain degree
$n$
$$ (X_*\odot Y_*)(\ell)_n = \bigoplus_{p+q=\ell} \bigoplus_{r+s=n}
k\Sigma_\ell \otimes_{k\Sigma_p \otimes k\Sigma_q} X_r(p) \otimes
Y_s(q).$$
Thus $$ (X_*\odot Y_*)(\ell) = \bigoplus_{p+q=\ell} k\Sigma_\ell
\otimes_{k\Sigma_p \otimes k\Sigma_q} X_*(p) \otimes Y_*(q)$$
if we follow the usual grading convention for tensor products of chain
complexes.

Similarly, for simplicial modules, $\smod$, and the corresponding
category of symmetric sequences therein, $\smod\Sigma$, we get for
$A_\bullet,B_\bullet \in \smod\Sigma$
$$ (A_\bullet\hodot B_\bullet)(\ell) = \bigoplus_{p+q=\ell}
k\Sigma_\ell \otimes_{k\Sigma_p \otimes k\Sigma_q} A_\bullet(p) \hot
B_\bullet(q)$$
where $\hot$ denotes the symmetric monoidal product for simplicial
modules, \ie, in simplicial degree $n$ this yields
$$ (A_\bullet(p) \hot B_\bullet(q))_n = A_n(p) \otimes B_n(q)$$
with a diagonal action of face and degeneracy operators.

As we will use these categories frequently in the rest of the paper we
ease notation by abbreviating $\dgmod\Sigma$ to $\sigmadg$ and $\smod\Sigma$ to
$\sigmas$.

We will make frequent use of the following two constructions. Let $M$
be a reduced object in $\modules\Sigma$, \ie, $M(0)=0$.

The free associative monoid generated by $M$ is
$$ T(M) = \bigoplus_{i\geq 0} M^{\odot i}$$
and the free commutative monoid generated by $M$ is
$$C(M) = \bigoplus_{i\geq 0} M^{\odot i}/\Sigma_i.$$
Sometimes, we need the reduced version of $C(M)$, so let $\bar{C}(M)$
be the free
commutative non-unital monoid in symmetric sequences generated by
$M$. Then $\bar{C}(M) = \bigoplus_{i\geq 1} M^{\odot
  i}/\Sigma_i$. Unravelling the definitions shows that these objects
deserve their names.

Note that elements in $C(M)$ behave like polynomials in
every level, but globally they can differ. Take for instance
as $M$ the symmetric sequence that is concentrated in level $1$ and is
equal to $k$ there. Then we can define an element $d$ in $C(M)$ by
$d(\ell) = [\mathrm{id_\ell}\otimes 1 \otimes \ldots \otimes 1] \in
M^{\odot \ell}/\Sigma_\ell(\ell)$. So $d$ is non-trivial in every
level and if we would like to assign a degree to this 'polynomial' we
could do that levelwise with the degree of $d(\ell)$ being $\ell$, but
this assignment does not give rise to any reasonable notion of global
degree. 

A monoid in $A$ in  $\modules\Sigma$ has a multiplication map
$\mu\colon A \odot A \ra A$. For a fixed level $\ell$ the map
$\mu(\ell)$ is a $\Sigma_\ell$-equivariant map 
$$ \mu(\ell) \colon \bigoplus k\Sigma_\ell \otimes_{k\Sigma_p \otimes
  k\Sigma_q} A(p) \otimes A(q) \ra A(\ell).$$ 
As the coset
$\Sigma_\ell/\Sigma_p \times \Sigma_q$ has the set of
$(p,q)$-shuffles, $\Sh(p,q)$, 
as a set of representatives, one can also define a monoid in symmetric
sequences of modules by declaring that there is a map
$\mu(\sigma)\colon A(p)\otimes A(q) \ra A(p+q)$ for every
$(p,q)$-shuffle $\sigma$ and that these maps satisfy certain coherence
conditions as spelled out in \cite[Definition 2.1]{Ro}. 
\begin{defn} \label{def:shufflealgs} 
  \begin{itemize}
\item[]
\item 
A monoid in $\modules\Sigma$ is called a \emph{shuffle algebra}. 
\item
A commutative monoid in $\modules\Sigma$ is a \emph{commutative
  shuffle algebra}. 
  \end{itemize}
\end{defn}
Before we move on to the differential graded and simplicial context we
give three examples of commutative shuffle algebras in the category of
(graded) modules. 

\begin{itemize}
\item 
Let $V$ be a $k$-module. Then we can define the symmetric sequence
generated by $V$ as 
$$ \mathrm{Sym}(V)(\ell) = V^{\otimes \ell}$$
where $\Sigma_\ell$ acts by permuting the tensor coordinates. 
Then $\mathrm{Sym}(V)$ is a commutative shuffle algebra despite the
fact that its underlying graded object is the tensor algebra generated
by $V$, \ie, the free associative algebra generated by $V$.  
\item
If $A_*$ is a graded commutative $k$-algebra, then Stover defines a
symmetric sequence $A_*^\pm$ \cite[p.~323]{St}  with
$$ A_*^\pm(\ell) = A_\ell$$ 
with the $\Sigma_\ell$-action given by the sign-action: 
$$ \sigma.a := \mathrm{sgn}(\sigma)a.$$ 
With this convention, $A_*^\pm$ is actually a commutative shuffle
algebra. Note that just placing $A_\ell$ in level $\ell$ does \emph{not}
define a commutative monoid. 
\item
If $\varepsilon\colon A \ra k$ is an augmented commutative unital
$k$-algebra, then we can define the symmetric sequence $gr^\Sigma(A)$
in $k$-modules 
by $gr^\Sigma(A)(\ell) = I^\ell/I^{\ell +1}$ where $I$ denotes the
augmentation ideal of $A$ and where $gr^{\Sigma}(A)(\ell)$ carries the trivial
$\Sigma_\ell$-action.  Then $gr^{\Sigma}$ actually defines a functor from
the category of  augmented commutative unital
$k$-algebras to commutative shuffle algebras. 
\end{itemize}

We also get interesting functors from commutative shuffle algebras to
ordinary (graded) commutative algebras. The following notion is a
variant of Stover's definition \cite[14.4]{St}. 
\begin{defn}
Let $A$ be an augmented commutative unital shuffle algebra. Then we
define $\Psi(A)$ to be the graded module with 
$$\Psi(A)_n = \begin{cases} 0, & n \text{ odd, } \\
A_m, & n=2m. \end{cases}$$
We let $\Psi(A)$ carry a symmetrized multiplication by setting 
$$a \cdot b := \sum_{\sigma \in \Sh(p,q)} \sigma \mu(\id \otimes a
\otimes b) = \sum_{\sigma \in \Sh(p,q)}\mu(\sigma \otimes a
\otimes b) $$
for homogeneous $a \in \Psi(A)_{2p} = A(p)$ and $b \in \Psi_{2q(A)}$. 
\end{defn}

In the following we use the abbreviation $\Sh(p;n)$ for the set of
shuffles $\Sh(\underbrace{p,\ldots,p}_n)$. This set carries an action
of the symmetric group on $n$ letters where an element $\sigma \in \Sigma_n$
acts via the precomposition with the corresponding block permutation
$\sigma^b \in \Sigma_{pn}$ which permutes the blocks $\{1,\ldots,p\}$,
$\{p+1,\ldots,2p\}$,$\ldots$,$\{p(n-1)+1,\ldots,pn\}$ as $\sigma$
permutes the numbers $\{1,\ldots,n\}$. We denote by $S(p;n)$ a set of
representatives for the quotient of $\Sh(p;n)$ by $\Sigma_n$. 

\begin{thm} \label{thm:pd}
For every augmented commutative shuffle algebra with $A(0)=k$, $\Psi(A)$
is a graded divided power algebra.  
\end{thm}
\begin{proof}

For any element $x\in \Psi(A)$ of positive degree $2p$ we define
$$\gamma_n(x) := \sum_{\sigma \in S(p;n)} \mu(\sigma \otimes
x^{\otimes n}). $$
As $x^{\otimes n}$ is invariant under the $\Sigma_n$-action, this is
well-defined and $n! \gamma_n(x) = x^n$. We have to show
that this definition gives a divided power structure on
$\Psi(A)$. 

\begin{itemize}
\item
The product $\gamma_n(x)\cdot \gamma_m(x)$ is equal to 
$$ \sum_{\sigma \in \Sh(np,mp)}\sum_{\tau_1\in S(p;n)} \sum_{\tau_2\in
S(p;m)} \mu^\sigma((\sigma\circ (\tau_1\oplus \tau_2)) \otimes
x^{\otimes (n+m)}).$$ 

On the other hand $\gamma_{n+m}(x)$ is 
$$\sum_{\xi \in S(p;n+m)}\mu^\Sigma(\xi\otimes x^{\otimes (n+m)}.$$ 
Composing elements in $\Sh(np,mp)$ with the sum of elements in
$\Sh(p;n)$ and
$\Sh(p;m)$ gives precisely the set
$\Sh(p;n+m)$. In $\gamma_n(x)\cdot \gamma_m(x)$
we devide out by the action of $\Sigma_n$ and $\Sigma_m$ whereas for
$\gamma_{n+m}(x)$ we build the quotient with respect to $\Sigma_{n+m}$
and therefore we obtain
\begin{equation} \label{eq:gammanplusm} 
\gamma_n(x) \cdot \gamma_m(x) = \binom{n+m}{n}\gamma_{n+m}(x).
\end{equation}

\item
The $n$-the divided power of a sum of two homogeneous elements of
positive degree $2p$ is 
\begin{align*} 
\gamma_{n}(x+y) = & \sum_{\sigma \in S(p;n)} \mu^\Sigma(\sigma \otimes
(x+y)^{\otimes n}) \\ 
= & \sum_{\sigma \in S(p;n)}\sum_{i=0}^n\sum_{\tau \in \Sh(i,n-i)}
\mu^\Sigma(\sigma \otimes \tau(x^{\otimes i} \otimes y^{\otimes
  (n-i)})). 
\end{align*}
For $\sum_{i=0}^n \gamma_i(x)\cdot\gamma_{n-i}(y)$ we obtain 
$$ \sum_{i=0}^n \sum_{\xi \in \Sh(ip,(n-i)p)}\sum_{\tau_1 \in
  S(p;i)}\sum_{\tau_2 \in S(p;n-i)} \mu^\Sigma((\xi\circ(\tau_1\oplus
\tau_2)) \otimes x^{\otimes i} \otimes y^{\otimes (n-i)}).$$
As 
$$ \mu^\Sigma(\sigma \otimes \tau(x^{\otimes i} \otimes y^{\otimes
  (n-i)})) = \mu^\Sigma((\sigma \circ \tau^b) \otimes x^{\otimes i}
\otimes y^{\otimes (n-i)})$$ 
we can again finish the comparison by a bijection of the indexing sets
in the two summations.  

\item
For the iteration of divided powers we express $\gamma_n(\gamma_m(x))$
as 
$$ \sum_{\sigma \in S(pm;n)} \sum_{\rho_1\in S(p;m)} \ldots
\sum_{\rho_n \in S(p;m)} \mu^\Sigma((\sigma\circ (\rho_1\oplus
\ldots\oplus \rho_n)) \otimes (x^{\otimes m})^{\otimes n}).$$ 
The indexing set is in bijection with the quotient of
$\Sh(p;mn)$ 
by $\Sigma_n$ acting by block permutations composed with $\Sigma_m
\times \ldots \times \Sigma_m$ whereas 
for $\gamma_{nm}(x)$ we get the quotient of the same set of shuffles
by the group $\Sigma_{mn}$ and therefore 
\begin{equation} \label{eq:iteration}
\gamma_n(\gamma_m(x)) = \frac{(nm)!}{n!(m!)^n} \gamma_{nm}(x).
\end{equation}

\item
If we apply $\gamma_n$ to a product of of two homogeneous elements $x
\in \Psi_{2p}(A)$ and $y\in \Psi_{2q}(A)$, then we obtain
$$ \sum_{\sigma in S(p+q;n)} \sum_{\tau_1 \in \Sh(p,q)} \sum_{\tau_2
  \in \Sh(p,q)} \mu^\Sigma((\sigma\circ (\tau_1\oplus \tau_2)) \otimes
(x\otimes y)^{\otimes n}).$$
In contrast to this we get 
$$ x^n\cdot \gamma_n(y) = \sum_{\xi\in \Sh(pn,qn)}\sum_{\rho_1\in
  \Sh(p;n)} \sum_{\rho_2 \in S(q;n)}
\mu^\Sigma(\xi \circ (\rho_1\oplus \rho_2)\otimes x^{\otimes n}
\otimes y^{\otimes n})$$
so the $x$- and $y$-terms appear in a different order. Using the
relation 
$$ \mu^\Sigma(\xi \circ (\rho_1\oplus \rho_2)\otimes x^{\otimes n}
\otimes y^{\otimes n}) = \mu^\Sigma(\xi \circ (\rho_1\oplus \rho_2)
\circ \chi \otimes (x \otimes y)^{\otimes n})$$
for the permutation $\chi$ that shuffles the $x$-terms next to the
$y$-terms we can again compare the two terms by a bijection of the
indexing sets and hence we obtain
\begin{equation} \label{eq:cartanformula}
\gamma_n(x\cdot y) = x^n \cdot \gamma_n(y)
\end{equation}
for elements in positive degrees. The multilinearity of the tensor
product also ensures that the above equation holds for $x$ in
$\Psi(A)_0 = A(0) = k$. 
\end{itemize}
\end{proof}

Note that $\Psi(\mathrm{Sym}(V))$ is the tensor algebra $T(V[2])$ on the
$k$-module $V[2]$ ($V$ concentrated in degree $2$) with the shuffle
multiplication. This also carries a divided power structure which is
made explicit in \cite[\S 5]{roby}.  

\section{Homology theories}
\subsection{Harrison homology}
Benoit Fresse introduces a version of Harrison homology
for reduced commutative monoids in the category of symmetric
sequences of $k$-modules by mimicking the classical definition, thus
it is defined as
the homology indecomposables of the bar complex \cite[6.8]{F2}: If $A$
is a reduced commutative monoid in symmetric sequences in $\modules$, then
$$ \harr_*(A) := H_*(\text{Indec}B(A)).$$

Here, working with reduced monoids corresponds to the setting where
one considers augmented monoids and takes coefficients in the ground
ring. In this sense Harrison homology as above corresponds to Harrison
homology of $k \oplus A$ with coefficients in $k$. 
In the following we assume that $M$ is a reduced symmetric sequence in
$\modules$
such that every $M(\ell)$ is $k$-projective. Fresse
proves \cite[Proposition 6.9]{F2} that Harrison homology, $\harr_*$,
of $\bar{C}(M)$ gives back the generators, if $M$ is levelwise
$k$-projective, \ie,
$$ \harr_*(\bar{C}(M)) \cong \Sigma M.$$

\subsection{Gamma homology}
Gamma homology \cite{RoWh} can be identified with $E_\infty$-homology
\cite[Theorem 9.5]{F2},
$H^{E_\infty}_*$, which in
turn can be defined as the stabilization of $E_n$-homology. If $A$ is
a non-unital commutative algebra (in $k$-modules or reduced symmetric
sequences of $k$-modules), then Fresse shows \cite{F2}
$$H^{E_\infty}_*(A) = \colim H^{E_n}_*(A) \cong \colim
H_*\Sigma^{-n}B^n(A) =: \Sigma^{-\infty}B^\infty(A).$$

We assume again that $M$ is a reduced symmetric sequence which is
degreewise $k$-projective.
Fresse calculates the homology of the bar construction
applied to $\bar{C}(M)$ \cite[Proposition 6.2]{F2} and we obtain by an
iteration of this result: 
$$H_*B^n\bar{C}(M) \cong \bar{C}(\Sigma^nM).$$
The isomorphism is induced by a map $\nabla_n$. For $\nabla_1\colon
\bar{C}\Sigma M \ra B\bar{C}M$ one considers the inclusion $M \ra
\bar{C}M$. This map
combined with the natural map from the suspension to the bar
construction yields $\Sigma M \ra B\bar{C}M$ and as $B\bar{C}M$ is
commutative, we
can extend this map to $C\Sigma M$ \cite[6.2]{F2}. Note that the case
$n=1$ is the Hochschild-Kostant-Rosenberg Theorem in disguise.

As $E_n$-homology of $\bar{C}(M)$ is isomorphic to
$H_*(\Sigma^{-n}B^n\bar{C}(M))$, an immediate consequence of \cite[6.3]{F2} is:
\begin{prop}
For every $M$ as above
$$ H^{E_n}_*(\bar{C}(M)) \cong \Sigma^{-n}\bar{C}(\Sigma^nM).$$
\end{prop}

With the help of this identification, we can show that Gamma-homology
of $\bar{C}(M)$ behaves like Harrison homology, up to a suspension.

\begin{thm}
For every reduced symmetric module $M$ which is levelwise
$k$-projective we have
$$ H\Gamma_*(\bar{C}(M)) = H^{E_\infty}_*(\bar{C}(M)) \cong M.$$
This isomorphism is natural in $M$.
\end{thm}
\begin{proof}
We have to understand the maps in the direct system for
$H^{E_\infty}_*(\bar{C}(M))$, \ie, we have to understand the squares of the
form
$$\xymatrix{
{\Sigma^{-n}\bar{C}(\Sigma^nM)} \ar@{.>}[r]^{\psi_n} \ar[d]_{\Sigma^{-n}\nabla_n} &
{\Sigma^{-(n+1)}\bar{C}(\Sigma^{(n+1)}M)} \ar[d]^{\Sigma^{-(n+1)}\nabla_{n+1}}\\
{H_*\Sigma^{-n}B^n(\bar{C}(M))} \ar[r]^{\sigma_n} &
{H_*\Sigma^{-(n+1)}B^{n+1}(\bar{C}(M))}
}$$
and in particular, we have to determine what the effect of the maps
$\psi_n$ is. The map $\sigma_n$ is the stabilization map from
$E_n$-homology to $E_{n+1}$-homology. It is induced by the canonical map
$$ \Sigma B^n\bar{C}(M) \ra B^{n+1}\bar{C}(M).$$
Let $s_{-n}(s_nm_1\cdot\ldots\cdot s_nm_\ell)$ denote an element in
$\Sigma^{-n}\bar{C}(\Sigma^nM)$. Under $\Sigma^{-n}\nabla_n$ it is sent to
$s_{-n}[m_1]_n \cdot\ldots\cdot[m_\ell]_n$ where $\cdot$ denotes the
product in $B^n$ and $[-]_n$ denotes an $n$-fold iterated bracket in
the bar construction. Therefore
$$ \sigma_n(\Sigma^{-n}\nabla_n(s_{-n}[m_1]_n
\cdot\ldots\cdot[m_\ell]_n)) = s_{-(n+1)}[[m_1]_n
\cdot\ldots\cdot[m_\ell]_n].$$
For a product of monomial length one we obtain
$$ \sigma_n(\Sigma^{-n}\nabla_n(s_{-n}[m_1]_n) =
s_{-(n+1)}[m_1]_{n+1}$$
and this is in the image of $\Sigma^{-(n+1)\nabla_{n+1}}$ with
$$\psi_n(s_{-n}s_nm_1) = \psi_n(m_1) = s_{-(n+1)}s_{n+1}m_1 \in
\Sigma^{-(n+1)}\bar{C}(\Sigma^{(n+1)}M).$$

Elements of higher monomial length cannot be in the image of
$\Sigma^{-(n+1)\nabla_{n+1}} \circ \psi_n$ for degree
reasons. Therefore the maps in the stabilization sequence
$$ \psi_n\colon  \Sigma^{-n}\bar{C}(\Sigma^{n}M) \ra
\Sigma^{-(n+1)}\bar{C}(\Sigma^{(n+1)}M) $$
are given by identifying the summand $\Sigma^{-n}\Sigma^nM = M \subset
\Sigma^{-n}\bar{C}(\Sigma^{n}M)$ with the summand
$\Sigma^{-(n+1)}\Sigma^{(n+1)}M = M \subset
\Sigma^{-(n+1)}\bar{C}(\Sigma^{(n+1)}M)$ and by projecting all other
summands to zero. Thus $M$ is the direct limit of the stabilization
process.
\end{proof}

\subsection{Andr\'e-Quillen homology}

We can define Andr\'e-Quillen homology for reduced commutative monoids
in symmetric sequences as usual: For every such $A$ there is a
standard free simplicial resolution
$$ \xymatrix@1{{\ldots} \ar@<1ex>[r] \ar@<-1ex>[r]  \ar@<-3ex>[r]
  \ar@<3ex>[r] & {\bar{C}^3(A)} \ar[r] \ar@<2ex>[r]
  \ar@<-2ex>[r] \ar[l] \ar@<-2ex>[l] \ar@<2ex>[l]   & {\bar{C}^2(A)}
  \ar@<1ex>[r]  \ar@<-1ex>[r]
\ar@<1ex>[l] \ar@<-1ex>[l]  & {\bar{C}(A)}  \ar[l]  }$$
and we define Andr\'e-Quillen homology of $A$ (with trivial coefficients) to be
$$ \AQ_*(A) = H_*(Q_a(\bar{C}^{\bullet +1}(A))$$
where $Q_a(-)$ denotes the module of indecomposables and the homology
is taken of the corresponding chain complex. Note that the canonical
inclusion $A \ra \bar{C}(A)$ is a section to the augmentation
$\bar{C}(A) \ra A$ that codifies the commutative monoid structure on
$A$.

Definition \ref{def:mcat-simplicial} will describe a model category
structure on  
commutative monoids in symmetric sequences of simplicial modules and
one can prove then that 
any free simplicial resolution, $P_\bullet$, 
gives rise to the same homology groups, \ie, $\AQ_*(A) =
H_*(Q_a(P_\bullet))$.

\subsection{Comparison}
We obtain the following comparison result. 
\begin{thm} \label{thm:aqharr}
Let $A$ be a reduced commutative shuffle algebra such
that every $A(\ell)$ is projective as a $k$-module, then
$$ \AQ_*(A) \cong \Sigma^{-1}\harr_*(A).$$
\end{thm}
\begin{proof}
Let $P_\bullet$ be the free simplicial resolution in the category of reduced
commutative monoids in symmetric sequences with
$P_t=\bar{C}^{t+1}(A)$. We consider the bicomplex with
$\text{Indec}(B(P_t))_s$ in bidegree $(s,t)$.

Taking homology in $s$-direction and using Fresse's result gives
$$ H_s(\text{Indec}(B(P_t))_* = \harr_s(P_t) = \begin{cases}
0, & s>0, \\
\Sigma Q_a(P_t), & s=0,
\end{cases}$$
and therefore
$$H_tH_s(\text{Indec}(B(P_t))_* \cong H_t\Sigma Q_a(P_\bullet) \cong
\AQ_{t-1}(A).$$

On the other hand, the section $s \colon A \ra P_0$ ensures that
$H_t\text{Indec}(B(P_\bullet))$ reduces to $\text{Indec}(B(A))$ in
degree $t=0$. Thus the total complex of our bicomplex calculates
the Harrison homology groups of $A$ and hence we obtain
$$ \AQ_{r}(A) \cong \harr_{r+1}(A) \text{ for all } r \geq 0.$$
\end{proof}

In a similar manner we can show that Andr\'e-Quillen homology is
isomorphic to $\Gamma$-homology:

\begin{thm} \label{thm:aqhgamma}
Let $A$ be a reduced commutative shuffle algebra such
that every $A(\ell)$ is projective as a $k$-module, then
$$ \AQ_*(A) \cong H^{E_\infty}_*(A).$$
\end{thm}
\begin{proof}
In this case we consider the bicomplex which is
$C^{E_\infty}_p(\bar{C}^{q+1}(A))$ in bidegree $(p,q)$. Its total
complex computes $E_\infty$-homology of $A$, but taking homology in
$p$-direction first yields
$$ H^{E_\infty}_p(\bar{C}^{q+1}(A)) \cong \bar{C}^{q}(A) \cong
Q_a(\bar{C}^{q+1}(A))$$
and thus if we then take homology in $q$-direction we obtain
Andr\'e-Quillen homology of $A$.
\end{proof}

Therefore, all three homology theories coincide for reduced
commutative monoids in symmetric sequences:
$$ \Sigma^{-1}\harr_*(A) \cong \AQ_*(A) \cong H^{E_\infty}_*(A).$$

\subsection{Hodge decomposition}
Let $k$ be any field and assume that $A$ is a non-unital
reduced commutative shuffle algebra. For any symmetric sequence $M$ we
denote by $\bar{C}_qM$ the $q$th homogeneous part of $\bar{C}M$, \ie,
$\bar{C}_qM = M^{\odot q}/\Sigma_q$. 
\begin{defn} For $q>0$ the $p$th $q$-Andr\'e-Quillen homology group of $A$ is
  defined as 
$$\AQ_{p}^{(q)}(A) := H_p\Sigma^{-1}\bar{C}_q\Sigma Q_a(\bar{C}^{\bullet+1}(A)).$$
\end{defn}
This definition should be compared to the definition of
$q$-Andr\'e-Quillen homology of an augmented commutative algebra with
coefficients in the ground ring in the usual setting as the homology of
$\Lambda^q Q_a(C^{\bullet+1}(A))$ (\cite[\S 3.5]{L},\cite[\S
6]{Qm}). Note that $\AQ_{p}^{(1)}(A) \cong \AQ_{p}(A)$. 

\begin{thm} \label{thm:hodge}
For every reduced commutative shuffle $k$-algebra $A$  we have 
$$ H^{E_1}_n(A) \cong \bigoplus_{p+q=n}\AQ_{p}^{(q)}(A).$$  
\end{thm}
\begin{proof}
We consider the bicomplex in symmetric sequences given by 
$\Sigma^{-1}B_p(\bar{C}^{q+1}A)$ in bidegree $(p,q)$. Taking vertical
homology first yields $\Sigma^{-1}B_p(A)$ and taking the $p$th horizontal
homology afterwards gives $H^{E_1}_p(A)$. Thus the spectral sequence
converges to the $E_1$-homology groups of $A$. 

If we take horizontal homology first we obtain
$$ H^{E_1}_q(\bar{C}^{p+1}A) = H_q(\Sigma^{-1}B\bar{C}(\bar{C}^p(A))) \cong
(\Sigma^{-1}\bar{C}(\Sigma \bar{C}^pA))_q \cong \Sigma^{-1} \bar{C}_q\Sigma 
Q_a(\bar{C}^{p+1}A).$$ 
Thus vertical homology of this yields $\AQ^{(q)}_{p}(A)$. Therefore
we obtain a spectral sequence 
$$ E^2_{p,q} = \AQ^{(q)}_{p}(A) \Rightarrow H^{E_1}_{p+q}(A).$$ 
We have to show that there are no higher differentials in this spectral
sequence. To this end we consider the map of bicomplexes induced by
$\nabla$:
$$\Sigma^{-1} \bar{C}_p\Sigma  Q_a(\bar{C}^{q+1}A) \cong
\Sigma^{-1}\bar{C}_p\Sigma \bar{C}^qA   
\ra  \Sigma^{-1}B_p\bar{C}^{q+1}A.$$ 
The left hand side carries only non-trivial vertical
differentials. The horizontal homology groups  
$$ \Sigma^{-1}\bar{C} \Sigma \bar{C}^qA \cong H_*(\Sigma^{-1}B\bar{C}^{q+1}A)$$ 
are isomorphic and we also get an isomorphism on the associated
$E^2$-terms. Thus we know that the homology groups of the associated 
total complexes agree and hence the $E^2$-term is already
isomorphic to the $E^\infty$-term. As we are working over a field,
there are no extension issues. 
\end{proof}
\section{Symmetric sequences in the dg and simplicial context}
From now on $k$ is an arbitrary commutative unital ring. 
For a chain complex $X_*$ and an integer $r \geq 0$ we denote by
$F^r(X_*)$ the symmetric sequence
$$F^r(X_*)(n) = \begin{cases}0, & n \neq r,\\ k[\Sigma_r] \otimes X_*,
  & n=r. \end{cases}$$
As usual we denote by $\mathbb{D}^n$ the $n$-disc complex, \ie, the
chain complex which
has $k$ in degrees $n$ and $n-1$ and whose only non-trivial
differential is the identity map; $\mathbb{S}^n$ is the $n$-sphere
complex which has $k$ as the only non-trivial entry in chain degree
$n$.

A standard result \cite[Theorem 11.6.1, 11.5]{H} turns the category of
symmetric sequences of non-negatively graded chain complexes,
 into a cofibrantly generated model category:
\begin{prop}
The category \sigmadg is a cofibrantly generated model category. Its
generating cofibrations are given by
$$I_\Sigma = \{F^r(\mathbb{S}^{n-1})  \ra F^r(\mathbb{D}^n), n \geq 1,
r\geq 0\}$$
and its generating acyclic cofibrations are
$$ J_\Sigma= \{F^r(0)=0 \ra F^r(\mathbb{D}^n), n \geq 1, r\geq 0\}. $$
A map $f\colon X_* \ra Y_*$ in $\sigmadg$ is a weak equivalence
(fibration) if its evaluations $f(n) \colon X_*(n) \ra Y_*(n)$ are
fibrations (weak equivalences)  in the model
category structure of
non-negatively graded chain complexes for all $n \geq 0$.
\end{prop}

Similarly, as the model category of simplicial $k$-modules is
cofibrantly generated with generating cofibrations and generating
acyclic cofibrations induced from the ones in the model structure of
simplicial sets, we can transfer the model structure to
symmetric sequences of simplicial $k$-modules and obtain the following
structure.

\begin{prop}
There is a cofibrantly generated model category structure on the
category of symmetric sequences in simplicial $k$-modules, $s\Sigma$,
such that a 
map $f\colon A_\bullet \ra B_\bullet$ is a weak equivalence
(fibration) if and only if $f(n)\colon A_\bullet(n) \ra B_\bullet(n)$
is a weak equivalence (fibration) for all $n \geq 0$.
\end{prop}

The Dold-Kan correspondence establishes an equivalence of categories
between non-negatively graded differential graded objects in an
abelian category $\mathcal{A}$ and simplicial objects in
$\mathcal{A}$. The category $\sigmas$ is the
category of simplicial objects in $\sigmamod$ and $\sigmadg$ is the category of
non-negatively graded chain complexes in $\sigmamod$. Hence we can
apply the Dold-Kan theorem:
\begin{prop}
The normalization functor induces an equivalence of categories
$$ N \colon \sigmas \ra \sigmadg.$$
\end{prop}
If we denote by $\Gamma$ the inverse of $N$, then the pair
$(N,\Gamma)$ gives rise to a Quillen equivalence between the model
categories $\sigmas$ and $\sigmadg$. As the abelian structure in
$\sigmamod$ is formed levelwise, we obtain
that for an $A_\bullet$ in $\sigmas$ the normalization of $A_\bullet$ is given by
$$ (N(A_\bullet))(n) = N(A_\bullet(n)).$$
Consequently we also get that $\Gamma$ is formed levelwise and
therefore if we start with a chain complex $X_*$ then
\begin{equation} \label{eq:gammafr}
\Gamma(F^rX_*) = F^r\Gamma(X_*)
\end{equation}
where $F^r$ is the functor that assigns to a simplicial $k$-module
$A_\bullet$ the symmetric sequence in simplicial modules with
$$(F^rA_\bullet)(n) = \begin{cases}0, & n \neq r,\\ k[\Sigma_r] \otimes A_\bullet,
  & n=r. \end{cases}$$
Note that we can express $k[\Sigma_r] \otimes A_\bullet$ also as
$$ \underline{k[\Sigma_r]} \hot A_\bullet $$
where $ \underline{k[\Sigma_r]}$ denotes the constant simplicial
object with value $k[\Sigma_r]$ in every simplicial degree and $\hot$
is the tensor product of simplicial modules.
\section{A model structure on reduced commutative monoids in
 symmetric sequences }

In characteristic zero there is a nice model structure on the category
of differential graded commutative algebras where the fibrations and weak
equivalences are determined by the forgetful functor to
(non-negatively graded) chain complexes. If the characteristic of the
field $k$ fails to be zero or if we want to work relative to a
commutative ground ring $k$, then such a model structure does not
exist. Stanley constructed a model structure with different features
in \cite{Stanley}.

One crucial problem is that the free commutative
algebra generated an acyclic complex doesn't have to be acyclic. For
instance for the $n$-disc complex, $\mathbb{D}^n$, for even $n$ the free
commutative algebra on $\mathbb{D}^n$ is $k[x_n] \otimes
\Lambda(x_{n-1})$. The differential is determined by
$\partial(x_n)=x_{n-1}$ and the derivation property. If we consider
$\partial(x_n^2)$ then this gives $2x_{n-1}x_n$, and if we can't
divide by $2$ in $k$, then this phenomenon and similar ones in higher
degrees cause non-trivial homology. We will see that such problems
disappear when one works with symmetric sequences. 

\begin{defn}
\begin{enumerate}
\item[]
\item
We denote the category of unital commutative monoids in symmetric sequences 
of non-negatively graded chain complexes by $C_\sigmadg$ and the 
one of unital commutative monoids in symmetric sequences of simplicial
modules by  $C_\sigmas$. 
\item
We call an object $A \in C_\sigmas$ (or $C_\sigmadg$) \emph{pointed},
if $A(0) =k$. The full  
subcategory consisting of these objects is denoted by $C^+_\sigmas$
(or $C^+_\sigmadg$).  
\item
The category of \emph{reduced commutative monoids in symmetric
  sequences} in $\sigmadg$ or $\sigmas$  
consists of commutative non-unital monoids $A$ with $A(0)=0$. We
denote by $C^-_\sigmadg$ and $C^-_\sigmas$ the corresponding
categories.  
\end{enumerate}
\end{defn}

Reduced differential graded commutative shuffle algebras avoid this
problem. Before we discuss a suitable model category structure we give
some examples of such algebras. 

\begin{itemize}
\item 
Let $C_*$ be a non-negatively graded chain complex of $k$-modules. The
symmetric sequence  
$$ \mathrm{Sym}(C_*)(\ell) = C_*^{\otimes \ell}$$
is a differential graded commutative shuffle algebra and we can reduce
it by setting its zero level to zero. 
\item
Let  $\varepsilon\colon A_* \ra k$ be an augmented differential graded
commutative unital $k$-algebra with augmentation ideal $I_*$. We
define a symmetric sequence $gr^\Sigma(A_*)$ as 
$$gr^\Sigma(A_*)(\ell) = I_*^\ell/I_*^{\ell +1}$$ 
but now we let $\Sigma_\ell$-act on $gr^\Sigma(A_*)(\ell)$ by the
signum-action. Then $gr^\Sigma(A_*)$ is a differential
graded commutative shuffle algebra and we can define a reduced version
by setting its zero level to zero. In fact $gr^\Sigma$ defines a functor from 
the category of  augmented differential graded commutative unital
$k$-algebras to  differential graded commutative shuffle algebras. 
\end{itemize}

We consider the free commutative monoid functor from $\sigmadg$ to the
category of commutative monoids in $\sigmadg$, $C_\sigmadg$,
$$ C\colon \sigmadg \ra C_\sigmadg, \qquad V_* \mapsto C(V_*) =
\bigoplus_{\ell\geq 0} V_*^{\odot \ell}/\Sigma_\ell.$$
There is a canonical map from the unit $I$ to $C(V_*)$ for any $V_*$
given by the inclusion into the summand for $\ell =0$. Note that for
any $k$-module $V$ we can identify $\mathrm{Sym}(V)$ with
$C(F^1V)$. 

A crucial auxiliary result is the following.
\begin{lem} \label{lem:acyclic}
Let $X_*$ be $\bigodot_{i\in S}\bigodot_{j\in T} C(F^{r_i}\mathbb{D}^{n_j})$ where
$S$ and $T$ are some arbitrary indexing sets and $S \ni r_i\neq 0
\neq n_j\in T$ for all $i,j$. Then the canonical
map
$$ I \ra  H_*(X_*)$$
is an isomorphism of symmetric sequences in graded $k$-modules.
\end{lem}
\begin{rem}
Note that the case $r=0$ is excluded: If $Y_*$ is a chain complex,
then $F^0Y_*$ is a symmetric
sequence concentrated in level zero and $C(F^0Y_*)$ is also
concentrated in level zero where it is the free graded commutative
algebra generated by $Y_*$, so in that case we can't expect the result
to hold.
\end{rem}
\begin{proof}
By definition $F^{r}\mathbb{D}^{n}$ is concentrated in level $r$ with
value $k\Sigma_r \otimes \mathbb{D}^{n}$ and  this implies
$$ (F^r\mathbb{D}^{n})^{\odot a}(m) \cong  \begin{cases}
k[\Sigma_{ar}] \otimes (\mathbb{D}^n)^{\otimes a}, & m=ar, \\
0, & \text{ otherwise. }
\end{cases}$$
Hence  $C(F^{r}\mathbb{D}^{n})$ is only non-trivial in levels
of the form $\ell = ar$  and
$$C(F^r(\mathbb{D}^n))(ar) \cong (k[\Sigma_{ar}] \otimes (\mathbb{D}^n)^{\otimes
  a})/\Sigma_a. $$
But as chain complexes
$$ (k[\Sigma_{ar}] \otimes (\mathbb{D}^n)^{\otimes a})/\Sigma_a \cong
\bigoplus_{\Sigma_{ar}/\Sigma_a} (\mathbb{D}^n)^{\otimes a}$$
with $\Sigma_{ar}/\Sigma_a$ denoting the coset of the subgroup
$\Sigma_a$ of $\Sigma_{ar}$ where $\Sigma_a$ permutes the $r$-blocks
of numbers of cardinatily $a$ in $\{1,\ldots,ar\}$. Therefore the
homology of $C(F^{r}\mathbb{D}^{n})$ is trivial for all levels $\ell >
0$ and the only contribution to homology
arises from $I \subset C(F^ r(X_*))$ in level zero.

Note that
$$\bigodot_{i\in S}\bigodot_{j\in T}
C(F^{r_i}\mathbb{D}^{n_j}) \cong C(\bigoplus_{i\in S}\bigoplus_{j\in
  T}  F^{r_i}\mathbb{D}^{n_j})$$
because $\odot$ is the categorical sum in $C_\sigmadg$ and $C(-)$ is
left adjoint to the forgetful functor from commutative monoids to
$\modules\Sigma$.
An induction shows that finite $\odot$-products of factors of the form
$C(F^{r}\mathbb{D}^{n})$ have homology isomorphic to $I$ and as
$\bigodot_{i\in S}\bigodot_{j\in T} C(F^{r_i}\mathbb{D}^{n_j})$ is a
colimit over finite $\odot$-products we get the claim.
\end{proof}

Recall that $C^+_\sigmadg$ denotes the full subcategory of $C_\sigmadg$
consisting of commutative monoids $A$ in $C_\sigmadg$ with $A(0)
=\mathbb{S}^0$. Note that the objects $C(F^r(X_*))$ are in  $C^+_\sigmadg$ for
all $r > 0$. We denote by $\sigmadgred$ the full subcategory of $\sigmadg$
consisting of reduced objects. Let $C^-_\sigmadg$ be the category of
commutative non-unital monoids $B$ in $(\sigmadg, \odot)$ with $B(0)
=0$. We obtain  $C^+_\sigmadg$ from  $C^-_\sigmadg$ by adding the unit
$\mathbb{S}^0$, in fact there is an equivalence of categories between
$C^+_\sigmadg$ 
and $C^-_\sigmadg$
$$\xymatrix@1{
{C^-_\sigmadg} \ar@<0.5ex>[r]^{(\phantom{A})_+} & {C^+_\sigmadg}
\ar@<0.5ex>[l]^{\bar{(\phantom{A} )}}
}$$
where for any $A \in C^+_\sigmadg$ the non-unital algebra $\bar{A}$
consists of the augmentation ideal of $A$ and where $B_+ = \mathbb{S}^0 \oplus
B$.

Categorical sums are straightforward. In the category $C^+_\sigmadg$
the sum of two objects is given by their $\odot$-product. For reduced
commutative monoids we obtain an induced structure.
\begin{lem}
Let $B_1,B_2$ be two objects in $C^-_\sigmadg$. Then their categorical sum
is given by the symmetric sequence $B_1 \diamond B_2$ with
$$ (B_1 \diamond B_2)(\ell) = B_1(\ell) \oplus B_2(\ell) \oplus
\bigoplus_{p+q=\ell,  p,q\geq 1} k[\Sigma_ \ell] \otimes_{k[\Sigma_p \times \Sigma_q]}
B_1(p) \otimes B_2(q). $$
\end{lem}
\begin{proof}
Note that $B_1 \diamond B_2$ is isomorphic to the augmentation ideal of
$((B_1)_+) \odot ((B_2)_+)$. As the augmentation ideal functor is part of an
equivalence of categories, it preserves sums. This proves the
universal property and also determines the multiplication on $B_1
\diamond B_2$ as the one that is inherited from  $((B_1)_+) \odot ((B_2)_+)$.
\end{proof}
The reduced sequence $0$ which consists of the zero module in every
level is a unit for $\diamond$ (compare \cite[p.~267]{AM}).
\begin{rem} \label{rem:acyclic}
Note that the proof of Lemma \ref{lem:acyclic} also gives that $
\Diamond_{i\in S}\Diamond_{j\in T} C(F^{r_i}\mathbb{D}^{n_j})$ has trivial homology.
\end{rem}
\begin{thm}
The category $C^-_\sigmadg$ has a Quillen model category structure 
such that a morphism is a  fibration (weak equivalence), if its
underlying map in $\sigmadg$ is a fibration (weak equivalence).
\end{thm}
\begin{proof}

Let $\bar{C}(X_*)$ be the reduced free commutative monoid generated by
$X_* \in \sigmadgred$,
$$ \bar{C}(X_*) = \bigoplus_{\ell > 0} (X_*)^{\odot
  \ell}/\Sigma_\ell. $$

We consider the following two sets.
$$ I_- := \{\bar{C}(F^r(\mathbb{S}^{n-1})) \ra
\bar{C}(F^r(\mathbb{D}^{n})); r, n \geq 1\}, $$
$$ J_- := \{0 = \bar{C}(F^r(0)) \ra
\bar{C}(F^r(\mathbb{D}^{n})); r, n \geq 1\}.  $$
We show that $C^-_\sigmadg$ is a cofibrantly generated model category
with generating cofibrations $I_-$ and generating acyclic cofibrations
$J_-$. The weak equivalences are the maps inducing quasi-isomorphisms
in each level. We use Hovey's criterion \cite[2.1.19]{Ho}.

The domains of our generators are small and the weak equivalences
satisfy $2$-out-of-$3$ and closure under retracts. We have to
understand maps with the right lifting property (RLP) with respect to $I_-$
and $J_-$, $I_-$-inj and $J_-$-inj.

A diagram like
$$ \xymatrix{
{\bar{C}(F^r(\mathbb{S}^{n-1}))} \ar[r] \ar[d] & {X} \ar[d]^f \\
{\bar{C}(F^r(\mathbb{D}^{n}))} \ar[r] & {Y}
}$$
is adjoint to the diagram
$$ \xymatrix{
{\mathbb{S}^{n-1}} \ar[r] \ar[d] & {UX(r)} \ar[d]^{Uf(r)} \\
{\mathbb{D}^{n}} \ar[r] & {UY(r)}
}$$
in the category of chain complexes. Here, $UX$ denotes the underlying
object in $\sigmadg$ of $X$. Thus the RLP is equivalent to $Uf(r)$
being an acyclic fibration in the category of chain complexes for all
$r\geq 1$.

Analogously we get that the RLP with respect to $J_-$ is equivalent to
$Uf(r)$ being a fibration of chain complexes for all $r \geq
1$. Therefore we obtain that $I_-$-inj equals the intersection of
$J_-$-inj with the class of weak equivalences.

It remains to show that $J_-$-cells are weak
equivalences and $I_-$-cofibrations. Remark \ref{rem:acyclic} ensures
that each building block of a $J_-$-cell object is acyclic and so are
directed limits of sums, thus we get the acyclicity.  By
definition $I_-$-cofibrations are the maps with the left
lifting property (LLP) with respect to the maps that have the RLP with
respect to $I_-$. Hence we are looking for maps with the LLP with
respect to maps $f$ with $Uf(r)$ being an acyclic fibration in chain
complexes for all $r \geq 1$. These are maps $g$ such that $Ug(r)$ is
a cofibration of chain complexes for all $r \geq 1$. The maps in $J_-$
satisfy this property. We showed that this property
is preserved by pushouts and (transfinite) composition preserves this
property as well.
\end{proof}

\begin{cor} \label{cor:mcat}
The category $C^+_\sigmadg$ possesses a model category structure.
\end{cor}
\begin{proof}
A morphism $f$ in $C^+_\sigmadg$ is a weak equivalence, fibration or
cofibration if and only if $I \oplus f$ is one.
\end{proof}
Note that the model structure on $C^+_\sigmadg$ is also cofibrantly
generated with
$$ I_+ := \{C(F^r(\mathbb{S}^{n-1})) \ra
C(F^r(\mathbb{D}^{n})); r, n \geq 1\}$$
as generating cofibrations and
$$ J_+ := \{I = C(F^r(0)) \ra
C(F^r(\mathbb{D}^{n})); r, n \geq 1\}$$
as generating acyclic cofibrations.

\section{A Dold-Kan correspondence for commutative shuffle
algebras}

It is a folklore result, that the model categories of reduced simplicial
commutative algebras over $\Q$, $\scom$,  and of reduced differential graded
commutative algebras over $\Q$, $\dcom$, are Quillen equivalent. A
proof is given in \cite[p.223]{Q69}. We adapt Quillen's argument to
the setting of commutative shuffle algebras and show that the
implementation of the symmetric groups into the monoidal structure
allows us to drop the characteristic zero assumption.

We discussed the model category structure on $\comsigmadg$ before. On
$\comsigmas$ we take Quillen's model structure on simplicial objects
in a nice category. As usual, free maps are crucial:
\begin{defn} \label{def:mcat-simplicial}
\begin{itemize}
\item[]
\item
A morphism $f\colon A_\bullet \ra B_\bullet$ in $\comsigmas$ is
\emph{free}, if there is a subsymmetric sequence in sets $Z_q$ of
$B_q$ such that $B_q = A_q \odot C(kZ_q)$ and $f_q \colon A_q \ra B_q$
is the inclusion of $A_q$ into this sum. In addition, the $Z_q$ are
closed under degeneracies in $B_\bullet$, \ie, $s_i(Z_q) \subset
Z_{q+1}$ for all degeneracies $s_i$ of $B_\bullet$ and for all $q$.
\item
An object $A_\bullet$ in $\comsigmas$ is \emph{free}, if the map from
the initial object $\underline{k}$ to $A_\bullet$ is free, \ie, if
$A_q \cong C(kZ_q)$ for all $q$ with $Z_q \subset A_q$ as above.
\end{itemize}
\end{defn}
Let $W$ denote the forgetful functor $W \colon \comsigmas \ra \sigmas$.
\begin{defn} (compare \cite[2.9]{Qm})
A morphism $f\colon A_\bullet \ra B_\bullet$ in $\comsigmas$ is
\begin{itemize}
\item
a weak equivalence, if $W(f)$ is a weak equivalence in $\sigmas$.
\item
a cofibration, if $f$ is a retract of a free map and
\item
a fibration if it has the right lifting property with respect to
acyclic cofibrations.
\end{itemize}
\end{defn}
With these definitions, $\comsigmas$ is a model category.

The normalization functor $N\colon \sigmas \ra \sigmadg$
passes to a functor  $N\colon
\comsigmas \ra \comsigmadg$:
\begin{lem} \label{lem:Nlaxsm}
The functor $N\colon \sigmas \ra \sigmadg$ is lax symmetric monoidal.
\end{lem}
\begin{proof}
Let $A_\bullet, B_\bullet$ be any two objects in $\sigmas$. We have to
show that the diagram  
$$ \xymatrix{
{N(A_\bullet) \odot N(B_\bullet)} \ar[r]^{s_{A,B}} \ar[d]^{tw} &
{N(A_\bullet \hodot B_\bullet)} \ar[d]^{N(tw)} \\
{N(B_\bullet) \odot N(A_\bullet)} \ar[r]^{s_{B,A}} & {N(B_\bullet \hodot A_\bullet)}
}$$
commutes for a suitable binatural map $s$. Here, $tw$ denotes the
corresponding symmetry isomorphism.
For a fixed level $\ell$ we define $s$ as the composite
$$ \xymatrix{{\bigoplus_{p+q=\ell} k\Sigma_\ell \otimes_{k\Sigma_p
    \otimes k\Sigma_q} \otimes N(A_\bullet(p)) \otimes
  N(B_\bullet(q))} \ar[d]^{\mathrm{id} \otimes sh_{A_\bullet(p),B_\bullet(q)}} \\
{\bigoplus_{p+q=\ell}  k\Sigma_\ell \otimes_{k\Sigma_p \otimes k\Sigma_q} \otimes
N(A_\bullet(p) \hot B_\bullet(q))} \ar[d]^{\cong} \\
{N(\bigoplus_{p+q=\ell}
\underline{k\Sigma_\ell}  \hot_{ \underline{k\Sigma_p}
    \hot \underline{k\Sigma_q} } A_\bullet(p) \hot B_\bullet(q)).}}$$
Here, $sh_{A_\bullet(p),B_\bullet(q)}$ denotes the ordinary shuffle
transformation of the simplicial modules $A_\bullet(p)$ and
$B_\bullet(q)$. 
 For a fixed pair $(p,q)$ with $p+q=\ell$ a homogeneous element
$[\sigma \otimes x \otimes y]$ with $\sigma\in \Sigma_\ell$, $x \in
N(A_\bullet(p))$ and $y \in N(B_\bullet(q))$ is sent via $s$ to
$[\sigma \otimes sh(x\otimes y)]$. If we twist first and then apply
$s$ we get $[\sigma\circ \chi(p,q) \otimes sh(y\otimes x)]$.  As the shuffle
transformation is lax symmetric monoidal, this is the image of
$[\sigma \otimes sh(x\otimes y)]$ under $N(tw)$.
\end{proof}
\begin{lem} \label{lem:leftadj}
The functor $N \colon \comsigmas \ra \comsigmadg$ possesses a left adjoint
$$L_N \colon \comsigmadg \ra \comsigmas.$$
\end{lem}
\begin{proof}
The construction is standard: If $X_*$ is a reduced object in
$\sigmadg$ then we define $L_N(C(X_*))$ as $C(\Gamma(X_*))$.
Every object $A_* \in \comsigmadg$ can be written as a coequalizer
$$\xymatrix{{ C(\overline{C(\bar{A}_*)})}  \ar@<0.5ex>[r] \ar@<-0.5ex>[r] &
  {C(\bar{A}_*)} \ar[r]
  & {A_*.}}$$
As a left adjoint, $L_N$ has to respect colimits and hence we define
$L_N(A_*)$ as the coequalizer of
$$ \xymatrix{{ C(\Gamma(\overline{C(\bar{A}_*)})) =
    L_NC(\overline{C(\bar{A}_*)})}  \ar@<0.5ex>[r]
  \ar@<-0.5ex>[r] & {L_NC(\bar{A}_*) = C(\Gamma(\bar{A}_*)).}}$$
\end{proof}
Transferring Quillen's sketch of proof \cite[p.~223]{Q69} to the
setting of pointed commutative shuffle algebras yields the following result.
\begin{thm} \label{thm:doldkancommsh}
The pair $(N,L_N)$ induces a Quillen equivalence between the model
categories $\comsigmadg$ and $\comsigmas$ for every commutative ground
ring $k$.
\end{thm}
Before we prove the theorem, we state a few lemmata. First, we need to
understand the
associated graded of a free commutative monoid.
\begin{lem} \label{lem:assgr}
Let $S_*$ be a symmetric sequence in graded sets with $S_*(0)=\varnothing$. Then
the associated graded of $C(kS_*)$ with respect to the filtration
coming from powers of the augmenation ideal $m \subset C(kS_*)$ is
isomorphic to $C(kS_*)$:
$$grC(kS_*) =  C(kS_*)/m \oplus m/m^2\oplus m^2/m^3 \oplus\ldots \cong C(kS_*).$$
\end{lem}
This result is the analog of the well-know fact that the
associated graded of a free commutative algebra is isomorphic to the
very same free commutative algebra.

As $S_*(0)=\varnothing$ we get that $grC(kS_*)(\ell) = \bigoplus_{i=0}^\ell
m^i/m^{i+1}$ and that $(kS_*^{\odot i}/\Sigma_i)(\ell) = 0$ for $i > \ell$.

\begin{proof}
The indecomposables of $C(kS_*)$ are $m/m^2 \cong kS_*$. The inclusion
map $m/m^2 \ra grC(kS_*)$ extends to a morphism of commutative monoids
$$ \xi \colon C(kS_*) \ra grC(kS_*).$$
In every level $\ell$ elements $p(\ell)$ in
$$C(kS_*)(\ell) = I(\ell) \oplus kS_*(\ell) \oplus (kS_*^{\odot
  2}/\Sigma_2)(\ell) \oplus \ldots $$
have only finitely many non-trivial summands and we denote
$p(\ell)$ by $(p_0(\ell),\ldots, p_\ell(\ell),0,\ldots)$ with $p_i(\ell)
\in   (kS_*^{\odot   i}/\Sigma_i)(\ell)$. The map $\xi$ is then given
by
$$\xi(\ell)(p_0(\ell),\ldots, p_\ell(\ell),0,\ldots) =
([p_0(\ell)],[p_1(\ell)],\ldots,[p_\ell(\ell)],0,\ldots)   $$
where $[p_i(\ell)]$ denotes the equivalence class of $p_i(\ell)$ with 
respect to $m^{i+1}$. 

As the intersection $(kS_*^{\odot i}/\Sigma_i) \cap m^{i+1}$ is zero,
the map $\xi$ is injective.

Let $z$ be an arbitrary element of $grC(kS_*)$. Then $z(\ell)$ has
finitely many non-trivial summands and we write
$$ z(\ell) = ([z_0(\ell)], [z_1(\ell)], \ldots, [z_\ell(\ell)],
0,\ldots).$$
Here, $z_i(\ell) \in m^i(\ell)$ and thus we know that
$z_i(\ell) \in \bigoplus_{r\geq i} (kS_*^{\odot r}/\Sigma_r)(\ell)$ so
we can express $z_i(\ell)$ as
$q_i^i(\ell) + \ldots + q_{\ell}^i(\ell)$ with
$q_j^i(\ell) \in (kS_*^{\odot j}/\Sigma_j)(\ell) \subset m^j$. Therefore
$$\xi(q_0^0(\ell),\ldots, q_\ell^\ell(\ell),0,\ldots) =
([q_0^0(\ell)],\ldots,[q_\ell^\ell(\ell)],0,\ldots) =
([z_0(\ell)],\ldots,[z_\ell(\ell)],0,\ldots). $$
\end{proof}
\begin{lem}
\begin{enumerate}
\item[]
\item
Let $C_*$ be a cell object in $C^+_\sigmadg$ and $I$ its augmentation ideal, then
$ grC_* \cong C(I/I^2)$ in $C^+_\sigmadg$.
\item
If $A_\bullet$ is a free object in $C^+_\sigmas$ and if
$\hat{I}$ is its augmentation ideal, then
$$ gr A_\bullet \cong C(\hat{I}/\hat{I}^2) \in C^+_\sigmas.$$
\end{enumerate}
\end{lem}
\begin{proof}
The underlying commutative monoids in symmetric sequences of graded
modules of $C_*$ and $A_\bullet$ are of the form $C(X_*)$ and
$C(Y_\bullet)$, where $X_*$ and $Y_\bullet$ are trivial in level zero
and are degreewise free $k$-modules. Thus, by Lemma \ref{lem:assgr} 
we know that the canonical maps
$$ C(I/I^2) \ra  grC_*$$
and
$$C(\hat{I}/\hat{I}^2) \ra gr A_\bullet$$
are isomorphisms of underlying commutative monoids in symmetric
sequences in graded $k$-modules. It remains to show that these
isomorphism are compatible with the differential on $C_*$ and the
simplicial structure maps of $A_\bullet$.

Let $[\sigma\otimes x_1
\otimes \ldots \otimes x_r]$ denote a generator in $(I/I^2)^{\odot
  r}/\Sigma_r)(\ell)$ (\ie, $\sigma  \in \Sigma_r$, $x_i \in
I/I^2(p_i)$ for some suitable $p_i$) and let $d$ be the differential
in $C_*$. Then
$$ d[\sigma\otimes x_1 \otimes \ldots \otimes x_r] = \pm \sum_{j=1}^r
[\sigma\otimes x_1 \otimes \ldots \otimes dx_j \otimes \ldots  \otimes x_r]$$
and the canonical map $\xi$ sends this element to $\pm
\sum_{j=1}^r\mu(\sigma\otimes x_1 \otimes \ldots \otimes dx_j \otimes
\ldots  \otimes x_r)$ where $\mu$ denotes the multiplication in
$C_*$. As $d$ is a derivation the latter is equal
to $d(\mu(\sigma\otimes x_1 \otimes \ldots \otimes x_r))$.

The argument for the simplicial structure maps is similar. We spell it
out for the face maps. The $i$th face map, $d_i$, sends a generator
$[\sigma\otimes x_1 
\otimes \ldots \otimes x_r]$ in  $((\hat{I}/\hat{I}^2)^{\odot
  r}/\Sigma_r)(\ell)$ to  $[\sigma\otimes d_ix_1
\otimes \ldots \otimes d_ix_r]$ and applying $\xi$ yields
$\mu(\sigma\otimes d_ix_1 \otimes \ldots \otimes d_ix_r)$. As
$A_\bullet$ is a simplicial monoid, this is equal to $d_i(\mu(\sigma\otimes x_1
\otimes \ldots \otimes x_r))$.
\end{proof}
\begin{lem}
If $C_* \in \comsigmadg$ is a cell object, then $L_NC_*$ is free.
\end{lem}
\begin{proof}
Recall that a cell object $C_*$ in $\comsigmadg$ is a sequential limit
of pushouts of the form
$$ \xymatrix{
{\bigodot_{r\in R} \bigodot_{d \in D} C(F^r\mathbb{S}^{d-1})}
\ar[r] \ar[d] & {C_*^n} \ar[d] \\
{\bigodot_{r\in R} \bigodot_{d \in D} C(F^r\mathbb{D}^{d})} \ar[r] & {C_*^{n+1}}
}$$
where the left vertical map is induced by the inclusions of spheres
into disks. It therefore suffices to show that one such map
$L_NC(F^r\mathbb{S}^{d-1}) \ra L_NC(F^r\mathbb{D}^{d})$ is free. This
works similar to Quillen's argument \cite[Proof of 4.4]{Q69}:
$L_NC(F^r\mathbb{S}^{d-1})$ is $CF^r\Gamma(\mathbb{S}^{d-1})$ and this
in turn can be identified with
$CF^r(\bar{k}\Delta^{d-1}/\partial\Delta^{d-1})$. Similarly the
simplicial model of the $d$-disc can be chosen as
$\Delta^d/\Lambda^d_d$ where $\Lambda^d_d$ is the $d$-horn of
dimension $d$, \ie, the simplicial set that is generated by all top
faces of $\id_d \in \Delta^d$ but the last one. The inclusion of
$\mathbb{S}^{d-1}$ into $\mathbb{D}^d$ can then be modelled by the map
$d_d\colon \Delta^{d-1} \ra \Delta^{d}$. We can then choose $Z_q$ to
be the symmetric sequence in sets that is concentrated in level $r$ and is
generated as a module by all simplices in  $\Delta^d/\Lambda^d_d$ that
are not in the image of $\Delta^{d-1}/\partial\Delta^{d-1}$ under
$d_d$.
\end{proof}
For an inductive step we need the following auxiliary result about
homology and free objects.
\begin{lem} \label{lem:retract}
For every $X_* \in \sigmadg$ the canonical map
$\beta\colon H_*(CX_*) \ra  H_*(NC(\Gamma(X_*)) = \pi_*C(\Gamma(X_*))$
is an isomorphism.
\end{lem}
\begin{proof}
Stover shows \cite[9.10]{St} that for any reduced symmetric
sequence $M$ in \modules \,  the free commutative algebra generated by
$M$, $C(M)$, embeds into the free associative algebra $T(M) =
\bigoplus_{n \geq 0} M^{\odot n}$ via a split inclusion, $j\colon C(M)
\ra T(M)$. Thus we can transfer Quillen's retract argument \cite[Proof
of 4.5]{Q69} to our context and consider the commutative square
$$
\xymatrix{
 {H_*(CX_*)} \ar@/_1ex/[d]_{j_*} \ar[r]^{\beta}&
{H_*NC(\Gamma(X_*))}\ar@/_1ex/[d]_{j_*}\\
{H_*(TX_*)} \ar@/_1ex/[u]_\varrho \ar[r]^{\beta'}& {H_*NT(\Gamma(X_*)).}
\ar@/_1ex/[u]_\varrho
}
$$
Here, $\varrho$ is induced by a splitting of $j$. In every level the
tensor algebra $TX_*$ consists of copies of tensor powers of
$X_*$. The Eilenberg-Zilber equivalence turns $\beta'$ into an
isomorphism. As $\beta$ is a retract of $\beta'$ it is an isomorphism as well.
\end{proof}
\begin{proof}[Proof of Theorem \ref{thm:doldkancommsh}]
This proof is an adaptation of \cite[Proof of 4.6]{Q69}.

Let $C_* \in \comsigmadg$ and denote by $\hat{I}$ the augmentation
ideal of $L_NC_*$ and by $I$ the augmentation ideal of $C_*$. The powers
of $\hat{I}$ filter $\hat{I}$ and this filtration respects the
multiplication: $\hat{I}^r\cdot \hat{I}^s \subset \hat{I}^{r+s}$. As
$N$ is lax monoidal, we get that $N\hat{I}^r\cdot N\hat{I}^s \subset
N\hat{I}^{r+s}$ and the unit of the adjunction $\eta\colon C_* \ra
NL_NC_*$ satisfies $\eta(I^r) \subset N\hat{I}^r$.

We consider the associated graded of $C_*$, $ grC_* = \bigoplus_{r \geq 0}
I^r/I^{r+1}$, and similarly $gr(L_NC_*) = \bigoplus_{r\geq
  0}\hat{I}^r/\hat{I}^{r+1}$.  
As $N$ is exact we obtain 
$$ Ngr(L_NC_*) = N(\bigoplus_{r\geq 0}\hat{I}^r/\hat{I}^{r+1}) \cong
\bigoplus_{r\geq 0} N(\hat{I}^r/N\hat{I}^{r+1}).$$ 

We denote by $\sf{triv}$ the full subcategory of $C^-_\sigmadg$
consisting of objects with trivial multiplication and by $i$ the
inclusion functor from $\sf{triv}$ to $C^-_\sigmadg$. Let  $A_*$ be an
object of $\sf{triv}$ and  and let $C_* \in \comsigmadg$. Then the
morphisms in $C^-_\sigmadg$ from $\bar{C}$ to $i(A_*)$ are precisely
the maps in $\sf{triv}$ from $\bar{C}/\bar{C}^2$ to $A_*$.

Similarly, the morphisms in $C^-_\sigmadg$ from $N\hat{I}$ to $i(A_*)$
are in bijection with the morphisms in $\sf{triv}$ from
$N\hat{I}/N\hat{I}^2$ to $A_*$ but as $N\hat{I}/N\hat{I}^2$ is
isomorphic to $N(\hat{I}/\hat{I}^2)$ and as the category $\sf{triv}$
is equivalent to the category of reduced symmetric sequences of chain
complexes we can identify this set of morphisms with the morphisms in
symmetric sequences of simplicial vector spaces from
$\hat{I}/\hat{I}^2$ to $\Gamma(A_*)$ and these in turn correspond to
maps in $C^-_\sigmas$ from $\hat{I}$ to $i(\Gamma(A_*))$. As $L_N$
is left adjoint to $N$ we finally get a bijection with the morphisms
from $I$ to $N(i(\Gamma(A_*)))$ and the latter is isomorphic to
$i(A_*)$. Therefore, $I/I^2$ and $\hat{I}/\hat{I}^2$ satisfy the same
universal property concerning maps from $I$ to $i(A_*)$ and hence
they are isomorphic.

Note that an adjunction argument also shows that $\hat{I}/\hat{I}^2
\cong \Gamma(I/I^2)$.
The induced map on the associated graded induced by the unit of the
$(L_N,N)$-adjunction is therefore of the form
$$ gr(\eta)\colon gr(C_*) \cong C(I/I^2) \ra Ngr(L_NC_*) \cong
N(C(\hat{I}/\hat{I}^2)) \cong NC\Gamma(I/I^2). $$
As it sends the generators $I/I^2$ to $\hat{I}/\hat{I}^2$ it is of the
form as $\beta$ in Lemma \ref{lem:retract} and thus it is a weak
equivalence. A levelwise $5$-lemma argument and an induction then
shows that $H_*(I/I^r)(\ell) \cong H_*(N(\hat{I}/\hat{I}^r))(\ell)$ for all $r
\geq 2$. If we fix an $\ell$, then -- as $I$ and $\hat{I}$ are reduced
-- for all  $r \geq \ell +1$ we get  $I^r(\ell)=\hat{I}^r(\ell) =
0$ and thus
$$H_*(I(\ell)) \cong H_*(I/I^r)(\ell) \cong
H_*(N(\hat{I}/\hat{I}^r))(\ell) \cong H_*N(\hat{I}),$$
so $\eta$ is a weak equivalence.
\end{proof}

\begin{rem}
Of course, it is natural to ask whether one can extend the result
above and establish a Quillen equivalence between (reduced) $E_\infty$-monoids
and commutative monoids in symmetric sequences, or more generally,
whether for (certain types of) operads $P$,  homotopy $P$-algebras and
$P$-algebras have equivalent homotopy categories. We plan to pursue
this question in future work.
\end{rem}

\section{Commutative $Hk$-algebra spectra}

Brooke Shipley proved \cite{shipley} that there is a chain of Quillen
equivalences between the model categories of $Hk$-algebra
spectra and differential graded $k$-algebras. This chain is derived from a
composite of functors from $Hk$-module spectra in symmetric
spectra  via the 
category of symmetric spectra in simplicial $k$-modules and
symmetric spectra in non-negatively graded chain complexes: 
$$ \xymatrix@1{
{Hk\text{-mod}} \ar[r]^{Z} & {\Sp^\Sigma(\smod)} \ar[r]^{\phi^*N} &
{\Sp^\Sigma(\dgmod).} } 
$$
Here, $\Sp^\Sigma(\smod)$ is the category of symmetric sequences in
simplicial $k$-modules that are modules over the commutative monoid
$\tilde{k}(\mathbb{S})$ with
$\tilde{k}(\mathbb{S})(\ell)$ being the simplicial free
$k$-module  generated by the non-basepoint simplices of the
$\ell$-sphere. Similarly, $\Sp^\Sigma(\dgmod)$ is the category of
symmetric sequences in non-negatively graded chain complexes of
$k$-modules with a module structure over the commutative monoid
$k[\bullet]$ with $k[\bullet](\ell) =
k[\ell]$ being the chain complex  with chain
group $k$ concentrated in chain degree
$\ell$ with trivial $\Sigma_\ell$-action. These categories of
symmetric spectra are symmetric monoidal with respect to the smash
product which is nothing but the coequalizer of the tensor product of
symmetric sequences where the action of the respective commutative
monoid on the left and right factor is identified. 

The functor $Z$ is defined
as 
$$ Z(M) = \tilde{k}(M)
\wedge_{\tilde{k}(Hk)} Hk$$ 
and 
$$\phi^*N(A_\bullet)(\ell) = N(A_\bullet(\ell))$$
with $k[\bullet]$-module structure given by the equivalences  
$$ \phi(\ell) \colon k[\ell] \simeq
N(\tilde{k}(\mathbb{S})(\ell)).$$

Shipley shows \cite[p.~372]{shipley} that $Z$ is strong symmetric
monoidal and that $\phi^*N$ is lax symmetric monoidal. 

\begin{prop} \label{prop:hkexamples} 
The forgetful functors $V_1\colon \Sp^\Sigma(\smod) \ra \sigmas$ and
$V_2\colon \Sp^\Sigma(\nonnegch) \ra \sigmadg$ are lax symmetric
monoidal. Hence every commutative $Hk$-algebra spectrum $A$ 
gives rise to a commutative simplicial shuffle algebra $V_1(ZA)$ and a
commutative differential graded shuffle algebra $V_2(\phi^*N(ZA))$. 
\end{prop}
\begin{proof}
For $A_\bullet, B_\bullet \in \Sp^\Sigma(\smod)$ there is a binatural
projection map 
$$ \pi_1(A_\bullet,B_\bullet) \colon V_1(A_\bullet) \hodot V_1(B_\bullet) = A_\bullet
\hodot B_\bullet \ra A_\bullet
\hodot_{\tilde{\mathbb{Z}}(\mathbb{S})} B_\bullet = A_\bullet \wedge B_\bullet.$$ 

As $\tilde{k}(\mathbb{S})$ is a commutative monoid and as
$\hodot$ is a symmetric monoidal structure, $\pi_1$ turns $V_1$ into a
lax symmetric monoidal functor. An analogous argument applies to
$V_2$. 
\end{proof}

\begin{rem}
If $A$ is a commutative $Hk$-algebra spectrum, then $V_1(ZA)$
is pointed if $ZA(0) = \underline{k}$, for instance if $A$ is
a square-zero extension $A= Hk \vee M$ where $M$ is an
$Hk$-module concentrated in positive levels. 
\end{rem}

\end{document}